\documentclass[a4paper,11pt,twoside]{amsart}
\usepackage[english]{babel}
\usepackage[utf8]{inputenc}

\usepackage[a4paper,inner=3cm,outer=3cm,top=4cm,bottom=4cm,pdftex]{geometry}
\usepackage{fancyhdr}
\usepackage{titlesec}
\usepackage{color}
\usepackage{bold-extra}
\titleformat{\section}{\normalfont\scshape\centering}{\thesection}{1em}{}
\titleformat{\subsection}{\bfseries}{\thesubsection}{1em}{}
  
\usepackage{comment}
\usepackage{graphics}
\usepackage{aliascnt}
\usepackage[pdftex,citecolor=green,linkcolor=red]{hyperref}

\usepackage{amsmath}
\usepackage{amsfonts}
\usepackage{amssymb}
\usepackage{amsthm}
\usepackage{comment}
\usepackage{mathtools}

\newtheorem{theorem}{Theorem}[section]
\newtheorem{corollary}[theorem]{Corollary}

\newtheorem{lemma}[theorem]{Lemma}
\newtheorem{proposition}[theorem]{Proposition}
\theoremstyle{definition}

\numberwithin{equation}{section}

\newcommand{\ord}{\textup{ord}}
\newcommand{\Gal}{\textup{Gal}}

\newcommand{\id}{\textup{id}}

\newcommand{\li}{\textup{li}}

\title{Equality of orders of a set of integers modulo a prime}

\date{}
\author{Olli J\"arviniemi}
\address{Department of Mathematics and Statistics, P.O. Box 68, 00014 Helsinki, Finland}
\email{olli.jarviniemi@helsinki.fi}

\begin{document}
\begin{abstract}
For finitely generated subgroups $W_1, \ldots , W_t$ of $\mathbb{Q}^{\times}$, integers $k_1, \ldots , k_t$, a Galois extension $F$ of $\mathbb{Q}$ and a union of conjugacy classes $C \subset \Gal(F/\mathbb{Q})$, we develop methods for determining if there exists infinitely many primes $p$ such that the index of the reduction of $W_i$ modulo $p$ divides $k_i$ and such that the Artin symbol of $p$ on $F$ is contained in $C$. The results are a multivariable generalization of H.W. Lenstra's work. As an application, we determine all integers $a_1, \ldots , a_n$ such that $\ord_p(a_1) = \ldots = \ord_p(a_n)$ for infinitely many primes $p$. We also discuss the set of those $p$ for which $\ord_p(a_1) > \ldots > \ord_p(a_n)$. The obtained results are conditional to a generalization of the Riemann hypothesis.
\end{abstract}

\maketitle

\section{Introduction}

H.W. Lenstra \cite{Lenstra} has considered the following problem:

Let $K$ be a global field, $W$ an infinite, finitely generated subgroup of $K^{\times}$, $F$ a finite Galois extension of $K$, $C$ a union of conjugacy classes of $\Gal(F/K)$, and $k$ a positive integer. Are there infinitely many primes $p$ such that the index of the reduction of $W$ modulo $p$ divides $k$, and the Artin symbol of $p$ on $F$ is contained in $C$?

Denote this set of primes $p$ by $M = M(K, F, C, W, k)$. Assuming a suitable generalization of the Riemann hypothesis (GRH),\footnote{In the context of the Artin conjecture, by GRH it is often meant that the zeros of the Dedekind zeta-functions of number fields having real part between $0$ and $1$ have real part $\frac{1}{2}$. This generalization of the Riemann hypothesis is sometimes referred to as the extended Riemann hypothesis (ERH). For $K$ a function field the required Riemann hypothesis has been proven, but for number fields it has been not.} Lenstra determines a necessary and sufficient condition for $M$ to be infinite. One way to formulate this condition is that $M$ is finite if and only if there is an obstruction at a finite level. This is discussed in more detail later.

The results are general, and can be applied, for example, to determine when there are infinitely many $p$ such that a given integer $a$ is a primitive root modulo $p$ (under GRH). The famous Artin's primitive conjecture is that this set is infinite for all $a$ not equal to $-1$ or a square. Furthermore, it is conjectured that this set has a positive natural density. Hooley \cite{Hooley} was the first to prove that the conjecture holds under GRH.  While the Artin's primitive conjecture is still a conjecture, Heath-Brown \cite{Heath} has unconditionally proven results such as that there are at most two primes $a$ for which there are only finitely many desired primes $p$. For a comprehensive survey on the conjecture and related problems, see \cite{Moree-Survey}.

There is still some room for generalizations of Lenstra's work. For example, a natural question on orders of integers modulo primes is ``are there infinitely many primes $p$ such that $\ord_p(2) = \ord_p(3)$?'', and this does not directly follow from Lenstra's results. 

The Schinzel-Wójcik problem asks to determine all integers $a_1, \ldots , a_k$ such that for infinitely many primes $p$ we have $\ord_p(a_1) = \ord_p(a_2) = \ldots = \ord_p(a_k)$. Schinzel and Wójcik \cite{Schinzel-Wojcik} solved the problem for $n = 2$. In this case, there exists infinitely many such $p$ as long as $|a_1|, |a_2| > 1$. Their argument is elementary, though highly nontrivial.

Pappalardi and Susa \cite{Pappalardi-Susa} have proven that under GRH the density of these $p$ exists for any $a_i$. However, their results do not provide a condition for the infinitude of these primes. They also prove (proposition 14) that if the numbers $a_i$ satisfy certain properties, there exists only finitely many such $p$. We prove that, assuming GRH, this is the only obstruction, and in other cases there are infinitely many desired $p$. If there are infinitely many such $p$, then they also have positive density.

While proving this type of results seems to be difficult unconditionally, GRH is not the only hypothesis whose assumption leads to progress. Wójcik \cite{Wojcik} has proved that under the Schinzel Hypothesis H there are infinitely many desired $p$ in the case $n = 3$ assuming that $|a_i| > 1$ and that the subgroup of $\mathbb{Q}^{\times}$ generated by $a_1, a_2$ and $a_3$ does not contain $-1$. Recently Anwar and Pappalardi \cite{Anwar-Pappalardi} have determined a necessary and sufficient condition for the infinitude of $p$ such that all of $a_i$ are primitive roots modulo $p$ under the Schinzel Hypothesis H. Matthews \cite{Matthews} has previously proven stronger density results for such $p$ under GRH. See \cite{Fouad} for further discussion of the problem under Schinzel's hypothesis.

Our strategy to solve the Schinzel-Wójcik problem under GRH is to develop a generalization of Lenstra's results for several groups $W_1, \ldots , W_t$ in the place of a single group $W$. We first present a brief overview of Lenstra's work and some results on Kummer-type extensions, after which we provide our generalization. We then apply our machinery to prove the following results.

\begin{theorem}
\label{thm:equal}
Assume GRH. Let $a_1, \ldots , a_k$ be rationals not equal to $-1, 0, 1$. There are infinitely many primes $p$ such that $\ord_p(a_1) = \ldots = \ord_p(a_k)$ if and only if at least one of the following statements is false:
\begin{enumerate}
\item There exists integers $e_i$ such that
$$\prod a_i^{e_i} = -1.$$
\item There exists integers $f_i$ with an odd sum such that
$$\prod a_i^{f_i} = 1.$$
\end{enumerate}
Furthermore, the density of such primes exists, and if there are infinitely many such primes, their density is positive.
\end{theorem}
We quickly prove the necessity of the conditions of the theorem. If such integers $e_i$ exist, and we have $\ord_p(a_i) = O$ for all $i$, then $O$ has to be even, as we have
$$(-1)^O = \prod (a_i^{e_i})^O \equiv 1 \pmod{p}.$$
Similarly, if such integers $f_i$ exist, we must have $O \equiv 1 \pmod{2}$, as otherwise we would have
$$1 = 1^{O/2} = \prod a_i^{f_iO/2} \equiv \prod (-1)^{f_i} \equiv -1 \pmod{p}.$$
In this light Theorem \ref{thm:equal} may be viewed to say that the only obstruction for equality of orders is via parities of orders.

The condition of Theorem \ref{thm:equal} holds if $a_i > 1$ for all $i$:

\begin{corollary}
\label{cor:equalPositive}
Assume GRH. Let $a_1, \ldots , a_k$ be positive rationals not equal to $1$. There are infinitely many primes $p$ such that $\ord_p(a_1) = \ldots = \ord_p(a_k)$. Furthermore, the density of such primes exists and is positive.
\end{corollary}

For positive $a_i$ we may choose all of $\ord_p(a_i)$ to be equal to $\frac{p-1}{2}$ above, and in general the index can be taken to be of the form $3 \cdot 2^s$ for a suitable $s$.

\begin{theorem}
\label{thm:ordering}
Assume GRH. Let $a_1, a_2, \ldots , a_t$ be rationals which are pairwise multiplicatively independent over $\mathbb{Q}$. There are infinitely many primes $p$ such that $\ord_p(a_1) > \ord_p(a_2) > \ldots > \ord_p(a_t)$. Furthermore, the density of such primes exists and is positive.
\end{theorem}

In other words, all $t!$ orderings of $\ord_p(a_i)$ are possible under the assumption of pairwise independence. Without this assumption the statement need not hold. A couple of counterexamples are
\begin{itemize}
\item $\ord_p(a^2) > \ord_p(a)$,
\item $\ord_p(a) > \ord_p(a^3) > \ord_p(a^2)$, and
\item $\ord_p(a) > \ord_p(b) > \ord_p(b^2) > \ord_p(a^2)$.
\end{itemize}
For more discussion on the necessity of the conditions, see Section \ref{sec:discussion}. The proof of Theorem \ref{thm:ordering} actually gives the following stronger statement.

\begin{theorem}
\label{thm:strongordering}
Assume GRH. Let $a_1, a_2, \ldots , a_k$ be rationals which are pairwise multiplicatively independent over $\mathbb{Q}$, and let $C > 1$ be a constant. There are infinitely many primes $p$ such that $\ord_p(a_i) > C\ord_p(a_{i+1})$ for all $1 \le i < k$. Furthermore, the density of such primes exists and is positive.
\end{theorem}

We give one more application.
\begin{theorem}
\label{thm:system}
Assume GRH. Let $a_1, \ldots , a_k, b_1, \ldots , b_k$ be arbitrary positive rationals not equal to $1$ and let $P_1, \ldots , P_m \in \mathbb{Z}[x]$ be arbitrary non-constant polynomials. There are infinitely many primes $p$ such that the equations
$$a_i^x \equiv b_i \pmod{p}, \ 1 \le i \le k$$
and
$$P_i(x) \equiv 0 \pmod{p}, 1 \le i \le m$$
are solvable. Furthermore, the density of such primes exists and is positive.
\end{theorem}
The result has the following interpretation: for any finite set of algebraic numbers (corresponding to the roots of $P_i$) and logarithms of positive integers (corresponding to the solutions of $a_i^{x_i} = b_i$) there exists infinitely many primes $p$ such that one has integer analogies of these numbers when performing arithmetic modulo $p$.

The theorems do not require GRH for all numbers fields. It suffices to assume the GRH for number fields obtained by adjoining roots of integers and unity to $\mathbb{Q}$, which we call Kummer-type extensions.

\section{Lenstra's work}

We only cover those parts of Lenstra's work which concern the number field case of the problem. Many of the details are omitted, and some are covered later when proving the generalization.

First, notation. Let $K, W, F, C, k$ and $M$ be as above. The letter $\ell$ will always denote a prime number. Let $q(\ell)$ be the smallest power of $\ell$ not dividing $k$. Define $L_{\ell} = \mathbb{Q}(\zeta_{q(\ell)}, W^{1/q(\ell)})$, where $\zeta_t$ is a $t$th primitive root of unity and $W^{1/n}$ denotes the set $\{w^{1/n}, w \in W\}$. For squarefree $n$, define $q(n) = \prod_{\ell \mid n} q(\ell)$ and $L_n$ to be the compositum of $L_{\ell}, \ell \mid n$. Define $C_n$ to be the set of $\sigma \in \Gal(FL_n/K)$ such that $\sigma|F \in C$ and $\sigma|L_{\ell} \neq \id_{L_{\ell}}$ for all $\ell \mid n$. Here $\sigma|L$ denotes the restriction of $\sigma$ on $L$, and $\id_L$ denotes the identity on $L$. Finally, define $d_n = \frac{|C_n|}{|\Gal(FL_n/K)|}$.

Clearly for all $n | m$ we have $d_n \ge d_m \ge 0$. Therefore, the numbers $d_n$ have a limit $d_{\infty}$ when $n$ ranges over the squarefree positive integers ordered by divisibility.\footnote{That is, for all $\epsilon > 0$ there exists an integer $N$ such that for all $n$ divisible by $N$ we have $|d_{\infty} - d_n| < \epsilon$.} 

The conjecture is that the density $d(M)$ of $M$ (with respect to the set of primes of $K$) equals $d_{\infty}$. The motivation is that for unramified $p$ we have $p \in M$ if and only if $(p|F) = C$ and $p$ splits in none of $L_{\ell}$, with $(p|F)$ being the Artin symbol.\footnote{If $p$ splits in $L_{\ell}$, then the index of $W$ is divisible by $q(\ell)$, which does not divide $k$. The other direction follows similarly. See Lenstra's work \cite{Lenstra} (Lemma 2.5) for details.} By the Chebotarev density theorem $d_n$ equals to the density of those $p$ for which this holds for all $\ell | n$, so by taking limits one would expect to have $d(M) = d_{\infty}$.

The case when $F = K$ and $C = \{\id_K\}$ of this conjecture has been dealt before, for which Lenstra refers to the work of Cooke and Weinberger \cite{Cooke}. The proof proceeds along the same lines as Hooley's \cite{Hooley} proof of Artin's primitive root conjecture (under GRH), and this is the only step of the proof requiring GRH. Lenstra then proves the general case by reducing to the case $F = K$ by an elementary argument.

Having proven $d(M) = d_{\infty}$, Lenstra focuses on determining a condition for the positivity of $d_{\infty}$. This is done in two parts. First, it is proven that if $d_n \neq 0$ for all $n$, then $d_{\infty} \neq 0$ (the converse being trivial). Then it is proven that if for a certain explicitly defined $H$ we have $d_H \neq 0$, then in fact $d_n \neq 0$ for all $n$. We do not cover the latter part.

To prove $d_n \neq 0$ for all $n$ implies $d_{\infty} \neq 0$, Lenstra proves that one has a product formula of the form 
$$d_{n\ell} = d_n\left(1 - \frac{1}{[L_{\ell} : K]}\right),$$
where $n$ is squarefree and $\ell \nmid n$ is large enough. Here $1 - \frac{1}{[L_{\ell} : K]}$ represents the density of primes not splitting in $K(\zeta_{q(\ell)}, W^{1/q(\ell)})$. This formula is obtained by proving that $L_{\ell}$ and $L_dF$ are linearly disjoint.

By this one gets
$$d_n = d_m \prod_{\ell | n, \ell > c} 1 - \frac{1}{[L_{\ell} : K]},$$
where $c$ is some constant and $m = \prod_{\ell | n, \ell \le c} \ell$. By taking limits one is left with proving that the infinite product
$$\prod_{\ell > c} 1 - \frac{1}{[L_{\ell} : K]}$$
converges to a strictly positive number. This follows from the bound $[L_{\ell} : K] \ge \ell(\ell - 1)$, which holds for large enough $\ell$ (see Proposition \ref{prop:kummerK} below).

\section{Kummer-type extensions}

Understanding extensions of the form
$$\mathbb{Q}(\zeta_n, a_1^{1/m_1}, \ldots , a_k^{1/m_k}),$$
where $a_i \in \mathbb{Q}, m_i \mid n$, is important both in Lenstra's method (in particular, proving that $L_{\ell}$ and $L_dF$ are linearly disjoint for $\ell$ large) and in applications of the method. Here we present results on these Kummer-type extensions, which are enough in many applications, including the ones we present. Related, more general results over number fields have been given in \cite{PeruccaSgobba}.

We first collect a couple of standard results. The proofs are omitted.

The first one concerns the compositums of Galois extensions. Extensions satisfying the properties are called linearly disjoint.
\begin{proposition}
Let $F_1$ and $F_2$ be Galois extensions of $K$. The following are equivalent.
\begin{itemize}
\item[(i)] $[F_1F_2 : K] = [F_1 : K][F_2 : K]$.
\item[(ii)] $F_1 \cap F_2 = K$.
\item[(iii)] There exists a $K$-basis of $F_1$ which is linearly independent over $F_2$.
\item[(iv)] $\Gal(F_1F_2/K) \cong \Gal(F_1/K)\times \Gal(F_2/K)$.
\end{itemize} 
\end{proposition}

\begin{proposition}
Any subfield of a cyclotomic field $\mathbb{Q}(\zeta_n)$ is Galois.
\end{proposition}

The Kronecker-Weber theorem:
\begin{proposition}
A finite Galois extension of $\mathbb{Q}$ is abelian if and only if it is a subfield of some (finite) cyclotomic field.
\end{proposition}

\begin{proposition}
\begin{itemize}
\item[(i)] If $p \equiv 1 \pmod{4}$ is a prime, then $\sqrt{p} \in \mathbb{Q}(\zeta_p)$.
\item[(ii)] If $p \equiv 3 \pmod{4}$ is a prime, then $\sqrt{-p} \in \mathbb{Q}(\zeta_p)$, and so $\sqrt{p} \in \mathbb{Q}(\zeta_{4p})$.
\item[(iii)] $\sqrt{2} \in \mathbb{Q}(\zeta_8)$.
\end{itemize}
\end{proposition}

We then present some results which are not as well-known. We start with radicals in cyclotomic fields.

\begin{proposition}
Let $a \in \mathbb{Q}$ be such that $|a|$ is not a perfect power in $\mathbb{Q}$. Then $a^{1/n}$ belongs to some cyclotomic field if and only if $n = 1$ or $n = 2$.
\end{proposition}
\begin{proof}
If-part follows from above. For the other part, assume $a > 0$ and that $a^{1/n}$ belongs to a cyclotomic field. Now $\mathbb{Q}(a^{1/n}) \subset \mathbb{R}$ is Galois, so the conjugates of $a^{1/n}$ are real. As the conjugates of $a^{1/n}$ are roots of $x^n - a$, either the minimal polynomial of $a^{1/n}$ is $x - a^{1/n}$ or $(x - a^{1/n})(x + a^{1/n}) = x^2 - a^{2/n}$. In both cases $n \le 2$.
\end{proof}
\begin{corollary}
\label{cor:rootCyclotomicUni}
Let $a \in \mathbb{Q}, 0 \neq n \in \mathbb{Z}$ be such that $a^{1/n}$ belongs to a cyclotomic field. Then $|a|^2$ is a perfect $n$th power in $\mathbb{Q}$.
\end{corollary}

We then provide a multivariable analogue to this corollary. Recall that non-zero rationals $a_1, \ldots , a_k$ are multiplicatively independent if the equation $a_1^{x_1} \cdots a_k^{x_k} = 1$ has only the solution $x_i = 0$.

\begin{proposition}
Let $a_1, \ldots, a_k$ be multiplicatively independent rationals. There exists an integer $N > 0$ with the following property: if $n, m_1, \ldots , m_k$ are integers such that $|a_1^{m_1} \cdots a_k^{m_k}|$ is a perfect $n$th power, then $n \mid m_iN$ for all $i$.
\end{proposition}

\begin{proof}
Let $p_1, \ldots , p_t$ be all of the primes which divide some numerator or denominator of $a_i$. Construct the matrix $A$ whose $i$th row $v_i$ consists of the exponents $v_{p_1}(|a_i|), \ldots , v_{p_t}(|a_i|)$ of the primes $p_1, \ldots , p_t$ in the prime factorization of $|a_i|$. By the assumption, the rows of $A$ are linearly independent. As the row rank equals the column rank, one may take some $k$ primes $q_1, \ldots , q_k \in \{p_1, \ldots , p_t\}$ such that the corresponding column vectors $w_1, \ldots , w_t$ are linearly independent.

By linear independence, let $c_1, \ldots , c_k \in \mathbb{Q}$ be such that
$$\sum_{i = 1}^k c_iw_i = (1, 0, 0, \ldots , 0).$$
Multiply by the product $N_1$ of the denominators of $c_i$ to get
$$\sum_{i = 1}^k C_iw_i = (N_1, 0, 0, \ldots , 0)$$
for $C_i \in \mathbb{Z}$.

Consider then the component
$$\sum_{j = 1}^k m_jv_{q_i}(|a_j|) \equiv 0 \pmod{n}.$$
of $q_i$ in the sum $\sum _iv_i$. Multiply by $C_i$ and sum over $i = 1, \ldots , k$. We get
\begin{align*}
0 \equiv \sum_{i = 1}^k C_i \sum_{j = 1}^k m_jv_{q_i}(|a_j|) \\
\equiv \sum_{j = 1}^k m_j \sum_{i = 1}^k C_iv_{q_i}(|a_j|) \\
\equiv m_1N_1 \pmod{n},
\end{align*}
the last equality following from the choice of $C_i$. Thus, $n \mid m_1N_1$.  Similar procedure for other indices gives the result.
\end{proof}

Combine this with Corollary \ref{cor:rootCyclotomicUni}:

\begin{proposition}
\label{prop:rootCyclotomicMulti}
Let $a_1, \ldots, a_k$ be multiplicatively independent rationals. There exists an integer $N > 0$ with the following property: if $n, m_1, \ldots , m_k$ are integers such that $a_1^{m_1/n} \cdots a_k^{m_k/n}$ belongs to some cyclotomic field, then $n \mid m_iN$ for  all $i$.
\end{proposition}

We then get to the Kummer-type extensions. The main ingredient is the following lemma.

\begin{proposition}
\label{prop:garret}
Let $n$ be a positive integer and let $K$ be an extension of $\mathbb{Q}$ containing $\zeta_n$. Let $a_1, \ldots , a_k$ non-zero be elements of $K$. Assume that $a_i/a_j$ is not an $n$th power in $K$ for any $i \neq j$. Then the elements $\sqrt[n]{a_1}, \ldots , \sqrt[n]{a_k}$ are linearly independent over $K$.
\end{proposition}

The statement and proof have been given in \cite{Garrett}. Due to the importantness of the result we give the short proof here.

\begin{proof}
Let $L$ be a Galois extension of $K$ containing all of $\alpha_j = \sqrt[n]{a_j}$. Assume the contrary, and consider the shortest linear combination $$\sum c_j\alpha_j = 0,$$
where $0 \neq c_j \in K$. Clearly there is at least two summands. Let $i$ and $j$ be two indices in the sum. 

As $\alpha_i/\alpha_j \not\in K$ by assumption, there exists some $\sigma \in \Gal(L/K)$ such that $\sigma(\alpha_i/\alpha_j) \neq \alpha_i/\alpha_j$, so
$$\frac{\sigma(\alpha_i)}{\alpha_i} \neq \frac{\sigma(\alpha_j)}{\alpha_j}.$$
We now have
$$0 = \frac{\sigma(\alpha_i)}{\alpha_i} \cdot 0 - \sigma(0) = \frac{\sigma(\alpha_i)}{\alpha_i}\sum_{t} c_t\alpha_t - \sigma\left(\sum_t c_t \alpha_t\right) = \sum_t c_t \alpha_t \left(\frac{\sigma(\alpha_i)}{\alpha_i} - \frac{\sigma(\alpha_j)}{\alpha_j}\right).$$
The coefficient of $\alpha_i$ in this sum is $0$ while that of $\alpha_j$ is not zero, so we have obtained a shorter linear combination equal to zero. This contradiction proves the result.
\end{proof}

We now get our first main tool on Kummer-type extensions. This is a special case of the results in \cite{PeruccaSgobba}.

\begin{proposition}
\label{prop:kummerQ}
Let $a_1, \ldots , a_k$ be multiplicatively independent rationals. There exists a constant $C > 0$ such that for any $n, m_1, \ldots , m_k$, where $m_i \mid n$ for all $i$, one has
$$[\mathbb{Q}(\zeta_n, a_1^{1/m_1}, \ldots , a_k^{1/m_k}) : \mathbb{Q}] \ge C\phi(n)m_1 \cdots m_k.$$
Furthermore, there exist a positive integer $N$ such that for all $n, m_1, \ldots , m_k$, where $m_i \mid n$ for all $i$, one has
$$[\mathbb{Q}(\zeta_{Nn}, a_1^{1/Nm_1}, \ldots , a_k^{1/Nm_k}) : \mathbb{Q}(\zeta_N, a_1^{1/N}, \ldots , a_k^{1/N})] = \frac{\phi(Nn)}{\phi(N)}m_1 \cdots m_k.$$ 
\end{proposition}

\begin{proof}
For the first part, let $N$ be as in Proposition \ref{prop:rootCyclotomicMulti}. By Proposition \ref{prop:garret}, the numbers of the form
$$a_1^{e_1/m_1} \cdots a_k^{e_k/m_k},$$
where the exponents $e_i$ range over $[0, m_i/N)$, are linearly independent over $\mathbb{Q}(\zeta_n)$, as the quotient of such numbers is not by the choice of $N$ contained in any cyclotomic field. Thus, we have
$$[\mathbb{Q}(\zeta_n, a_1^{1/m_1}, \ldots , a_k^{1/m_k}) : \mathbb{Q}] \ge \phi(n)\frac{m_1m_2 \cdots m_k}{N^k},$$
as desired.

For the second part, we first try to pick $N = N_0 = 1$. If this does not work, then we can take some integer $N_1$ divisible by $N_0$ such that
$$[\mathbb{Q}(\zeta_{N_1}, a_1^{1/N_1}, \ldots , a_k^{1/N_1}) : \mathbb{Q}(\zeta_{N_0}, a_1^{1/N_0}, \ldots , a_k^{1/N_0})] \le \frac{1}{2} \frac{\phi(N_1)}{\phi(N_0)}\left(\frac{N_1}{N_0}\right)^k.$$
If this $N_1$ does not work either, we may take some $N_2$ divisible by $N_1$ such that
$$[\mathbb{Q}(\zeta_{N_2}, a_1^{1/N_2}, \ldots , a_k^{1/N_2}) : \mathbb{Q}(\zeta_{N_1}, a_1^{1/N_1}, \ldots , a_k^{1/N_1})] \le \frac{1}{2} \frac{\phi(N_2)}{\phi(N_1)}\left(\frac{N_2}{N_1}\right)^k.$$
Continue in this manner. Assuming we can construct all of the numbers $N_0, N_1, \ldots , N_t$ for some $t$, by collapsing the tower of extensions with the tower law we get
$$[\mathbb{Q}(\zeta_{N_t}, a_1^{1/N_t}, \ldots , a_k^{1/N_t}) : \mathbb{Q}] \le \frac{1}{2^t}\phi(N_t)N_t^k.$$
By the first part of the proposition, one cannot pick $t$ arbitrarily large here, which proves the second part.
\end{proof}

The result immediatelly generalizes over any finite extension of $\mathbb{Q}$.
\begin{proposition}
\label{prop:kummerK}
Let $a_1, \ldots , a_k$ be multiplicatively independent rationals, and let $K$ be a finite extension of $\mathbb{Q}$. There exists a constant $C > 0$ such that for any $n, m_1, \ldots , m_k$, where $m_i \mid n$ for all $i$, one has
$$[K(\zeta_n, a_1^{1/m_1}, \ldots , a_k^{1/m_k}) : K] \ge C\phi(n)m_1 \cdots m_k.$$
Furthermore, there exist a positive integer $N$ such that for all $n, m_1, \ldots , m_k$, where $m_i \mid n$ for all $i$, one has
$$[K(\zeta_{Nn}, a_1^{1/Nm_1}, \ldots , a_k^{1/Nm_k}) : K(\zeta_N, a_1^{1/N}, \ldots , a_k^{1/N})] = \frac{\phi(Nn)}{\phi(N)}m_1 \cdots m_k.$$ 
\end{proposition}

\begin{proof}
The first result follows from Proposition \ref{prop:kummerQ} by the tower law. The proof of the second part is similar to that of Proposition \ref{prop:kummerQ}.
\end{proof}

As an important consequence of Proposition \ref{prop:kummerGalois} we get that a certain Galois group is the ``maximal possible''.
\begin{proposition}
\label{prop:kummerGalois}
Let $a_1, \ldots , a_k$ be multiplicatively independent rationals, and let $K$ be a finite Galois extension of $\mathbb{Q}$. There exists a positive integer $N$ with the follwing property:

For any integers $n, m_1, \ldots , m_k$, where $m_i \mid n$ for all $i$, and $x, x_1, \ldots , x_k$ with $(x, Nn) = 1, N \mid x-1, x_1, \ldots , x_k$ there exists an element of the Galois group of
$$K(\zeta_{Nn}, a_1^{1/Nm_1}, \ldots , a_k^{1/Nm_k})/K$$
sending
$$\zeta_{Nn} \to \zeta_{Nn}^{x}, a_i^{1/n} \to \zeta_{Nm_i}^{x_i}a_i^{1/n}.$$
\end{proposition}

Here is another consequence of the results.
\begin{proposition}
\label{prop:kummerDisjoint}
Let $a_1, \ldots , a_k$ be multiplicatively independent rationals, and let $K$ be a finite Galois extension of $\mathbb{Q}$. There exists an integer $N$ such that for any $n, n', m_1, \ldots m_k$, where $(n, Nn') = 1$ and $m_i \mid n$ for all $i$, the fields
$$\mathbb{Q}(\zeta_n, a_1^{1/m_1}, \ldots , a_k^{1/m_k})$$
and
$$K(\zeta_{Nn'}, a_1^{1/Nn'}, \ldots , a_k^{1/Nn'}).$$
are linearly disjoint and the former extension has degree $\phi(n)m_1 \cdots m_k$.
\end{proposition}

We conclude this section with a result on the maximal abelian extension of a Kummer-type extension.

\begin{proposition}
\label{prop:kummerAbelian}
Let $a_1, a_2, \ldots , a_k$ be arbitrary rationals. There exists rationals $b_1, \ldots , b_K$ such that for any $n, m_1, m_2, \ldots , m_k$, where $m_i \mid n$ for all $i$, the largest abelian subfield of
$$\mathbb{Q}(\zeta_n, a_1^{1/m_1}, \ldots , a_k^{1/m_k})$$
is a subfield of
$$\mathbb{Q}(\zeta_n, \sqrt{b_1}, , \ldots , \sqrt{b_K}).$$
\end{proposition}

\begin{proof}
We choose the numbers $b_1, \ldots , b_K$ to consist of all of the primes which divide some numerator or denominator of $a_1, \ldots, a_k$. Now it suffices to check that any extension $A/B$, where $A$ is an abelian number field and 
$$B = \mathbb{Q}(\zeta_n, \sqrt{b_1}, \ldots , \sqrt{b_K}),$$
is linearly disjoint with the extension $C/B$, where
$$C = B(\zeta_n, a_1^{1/m_1}, \ldots , a_k^{1/m_k}).$$
As any abelian number field $A$ is a subfield of a cyclotomic field, it suffices to do this for $A$ cyclotomic.

We prove linear disjointness by proving that a certain $B$-basis of $C$ is linearly independent over $AB$, too. Since the products of the form
$$a_1^{e_1/m_1} \cdots a_k^{e_k/m_k}, 0 \le e_i < m_i$$
span $C$ as a $B$-vector space, we may pick some subset of them, say $S$, which forms a $B$-basis for $C$.  Now the quotient of any two elements of $S$ does not belong to $B$. By X.Y, it suffices to prove that no quotient of two elements of $S$ belongs to $AB$.

Assume $q = a_1^{e_1/m_1} \cdots a_k^{e_k/m_k}$ is a quotient of two elements of $S$ belonging to $AB$. Write $q$ as
$$b_1^{f_1/m}b_2^{f_2/m} \cdots b_K^{f_K/m}, f_i \in \mathbb{Z}$$
up to a multiplication of a root of unity which is contained in $B$. Such a number is contained in a cyclotomic field if and only if $m \mid 2f_i$ for all $i$ by Corollary \ref{cor:rootCyclotomicUni}. In this case, $q$ is in fact contained in $B$.
\end{proof}

\section{Generalization for several groups}

The notation is similar to that of Lenstra. Let $W_1, \ldots , W_t$ be infinite, finitely generated subgroups of $\mathbb{Q}^{\times}$. Let $k_1, \ldots , k_t$ be positive integers, $F$ a finite Galois extension of $\mathbb{Q}$, and $C$ a conjugacy class of $\Gal(F/\mathbb{Q})$. We are interested in the set $M$ 
of primes $p$ for which the index of the reduction of $W_i$ modulo $p$ divides $k_i$ for all $i$, and for which $(p|F) \in C$. The letters $p$ and $\ell$ always denote prime numbers.

For each $1 \le i \le t$, let $q_i(\ell)$ be the smallest power of $\ell$ not dividing $k_i$. For each $\ell$ define
$$L_{\ell} = \mathbb{Q}(\zeta_{\max(q_1(\ell), \ldots , q_t(\ell))}, W_1^{1/q_1(\ell)}, \ldots , W_t^{1/q_t(\ell)}).$$
Let $L_n$ be the compositum of $L_{\ell}, \ell | n$. Let $C_n$ be the set of $\sigma \in \Gal(FL_n/\mathbb{Q})$ such that $(\sigma|F) \in C$ and such that $\sigma$ is not the identity on $\mathbb{Q}(\zeta_{q_i(\ell)}, W_i^{1/q_i(\ell)})$ for any $1 \le i \le t$ and $\ell | n$. Let $d_n = \frac{|C_n|}{|\Gal(FL_n/\mathbb{Q})|}$. As with the case $t = 1$ of one group, we have $d_n \ge d_m \ge 0$ for all $n | m$, and therefore the limit $d_{\infty}$ of $d_n$ exists when $n$ goes through the squarefree positive integers ordered by divisibility.

In the next two subsections we prove the following theorems.

\begin{theorem}
\label{thm:densitylimit}
Assume GRH. The density $d(M)$ of $M$ exists, and we have $d(M) = d_{\infty}$.
\end{theorem}

\begin{theorem}
\label{thm:densitynonzero}
If $d_n \neq 0$ for all squarefree positive integers $n$, then $d_{\infty} \neq 0$.
\end{theorem}

\subsection{Proof of Theorem \ref{thm:densitylimit}: $d(M) = d_{\infty}$}

As in Lenstra's work (\cite{Lenstra}, Lemma 3.2), we may reduce to the case $F = \mathbb{Q}$ and $C = \{\id_{\mathbb{Q}}\}$.

The proof is a modififcation to Hooley's \cite{Hooley} proof for Artin's primitive root conjecture. The works of Cooke and Weinberger \cite{Cooke} and Matthews \cite{Matthews} have also provided inspiration.

Let $R(\ell, p)$ be the statement ``$p$ splits in $\mathbb{Q}(\zeta_{q_i(\ell)}, W_i^{1/q_i(\ell)})$ for at least one index $1 \le i \le t$''. Let $N(x, \delta)$ be the number of $p \le x$ such that $R(\ell, p)$ is false for all $\ell \le \delta$. We want to prove that $N(x, x-1) = |M \cap [1, x]|$ tends to infinity with $x$ with rate $d_{\infty}\pi(x)$, where $\pi(x)$ denotes the number of primes $\le x$. Let $P(x, k)$ be the number of $p \le x$ such that $R(\ell, p)$ is true for all $\ell | k$. By inclusion-exclusion we have
$$N(x, \delta) = \sum_k \mu(k)P(x, k),$$
where the sum goes through all $k$ whose all prime divisors are $\le \delta$.

Let $\xi_1 = \frac{1}{6}\log(x)$, $\xi_2 = x^{1/2}/\log(x)^2$, and $\xi_3 = x^{1/2}\log(x)$. Let $L(x, \eta_1, \eta_2)$ be the number of $p \le x$ such that $R(\ell, p)$ is true for at least one prime $\eta_1 \le \ell \le \eta_2$. Now
$$N(x, x-1) = N(x, \xi_1) + O(L(x, \xi_1, x-1)),$$
and
$$L(x, \xi_1, x-1) \le L(x, \xi_1, \xi_2) + L(x, \xi_2, \xi_3) + L(x, \xi_3, x-1).$$

We first prove that $N(x, \xi_1)$ grows asymptotically as $d_{\infty}\pi(x)$, after which we will show that $L(x, \xi_1, x-1)$ is small.

Let $\mathcal{P}(y)$ be the product of primes at most $y$. We have
$$N(x, \xi_1) = \sum_{k \mid \mathcal{P}(\xi_1)} \mu(k)P(x, k),$$
For any fixed $m$, the sum
$$\sum_{k | \mathcal{P}(m)} \mu(k)P(x, k)$$
is asymptotically $d_{\mathcal{P}(m)}\pi(x)$. Thus, for $m \to \infty$ this approaches $d_{\infty}\pi(x)$, which is the desired claim. However, we need the error term with $m = \xi_1$ to be $o(x/\log(x))$. This indeed is the case (under GRH), as we will now show. 

Let $W_S$ be the subgroup of $\mathbb{Q}^{\times}$ generated by $W_i, i \in S$, where $S$ is an arbitrary subset of $\{1, 2, \ldots , t\}$. To $W_S$ we associate the fields
$$L_{\ell, S} := \mathbb{Q}(\zeta_{\max\{q_i(\ell) | i \in S\}}, \{W_i^{1/q_i(\ell)}, i \in S \}),$$
and denote by $L_{n, S}$ the compositum of $L_{\ell, S}$ with $\ell \mid n$. We have $L_{n, \emptyset} = \mathbb{Q}$ for all $n$.

$N(x, \xi_1)$ can be calculated by inclusion-exclusion as
\begin{align}
\label{eq:incExc}
\sum_{k | \mathcal{P}(\xi_1)} \sum_{S \subset \{1, 2, \ldots , t\}} |\{p : p \le x, p \text{ splits in } L_{k, S}\}|\mu(k)^{|S|}.
\end{align}
By Theorem 1.4 of \cite{Cooke} we have
$$|\{p : p \le x, p \text{ splits in } L_{k, s}\}| = \frac{\li(x)}{[L_{k, S} : \mathbb{Q}]} + O\left(x^{1/2}\log(x\prod_{i \in S} q_i(k))\right)$$
under GRH. Here $\li(x)$ denotes the logarithmic integral $\int_0^x \frac{1}{\log(t)} dt$.

In the sum \eqref{eq:incExc} we have $2^{\pi(\xi_1)+t} = O(x^{\epsilon})$ summands, so the error term for the whole sum becomes $O(x^{1/2 + \epsilon}\log(\mathcal{P}(\xi_1)x))$ using the fact $q_i(\ell) = \ell$ for large enough $\ell$ and all $i$. As we have
$$\log(\mathcal{P}(\xi_1)) = \sum_{p \le \xi_1} \log(p) = O(\xi_1 \log(\xi_1)) = O((\log x)^2),$$
the error term for $N(x, \xi_1)$ is $o(x/\log x)$.

We are left with proving that $L(x, \xi_1, x-1) \le L(x, \xi_1, \xi_2) + L(x, \xi_2, \xi_3) + L(x, \xi_3, x-1)$ is small. We do this term-by-term.

\begin{lemma}
\label{lem:Lbounds}
We have $L(x, \xi_2, \xi_3) = o(x/\log(x))$ and $L(x, \xi_3, x-1) = o(x/\log(x))$.
\end{lemma}

\begin{proof}
We have
$$L(x, \xi_2, \xi_3) \le \sum_{\xi_2 \le l \le \xi_3} P(x, \ell).$$
For $P(x, \ell)$ we have the inequality
$$P(x, \ell) \le |\{p : p \le x, p \equiv 1 \pmod{\ell}\}|,$$
as $p$ splitting in $\mathbb{Q}(\zeta_{q_i(\ell)}, W_i^{1/q_i(\ell)})$ means $p$ splits in $\mathbb{Q}(\zeta_{\ell})$, and therefore $p \equiv 1 \pmod{\ell}$. By Brun-Titchmarsh we have
$$|\{p : p \le x, p \equiv 1 \pmod{\ell}| \le \frac{2x}{(\ell-1)\log(x/\ell)} = O\left(\frac{x}{\ell\log(x)}\right),$$
so
$$\sum_{\xi_2 \le \ell \le \xi_3} P(x, \ell) = O\left( \frac{x}{\log(x)}\sum_{\xi_2 \le \ell \le \xi_3} \frac{1}{\ell} \right) = o\left(\frac{x}{\log(x)}\right).$$

For $L(x, \xi_3, x-1)$ we note that if $p$ is counted by $L(x, \xi_3, x-1)$, then $R(\ell, p)$ is true for some $\ell \ge \xi_3$, so $p$ splits in some $\mathbb{Q}(\zeta_{q_i(\ell)}, W_i^{1/q_i(\ell)})$.  This means that for any $w \in W_i$ we have $w^{(p-1)/q_i(\ell)} \equiv 1 \pmod{p}$. Fix some such $w = \frac{a}{b} > 1$. Now, any $p$ counted by $L(x, \xi_3, x-1)$ by the index $i \in \{1, \ldots , t\}$ divides
$$\prod_{m \le (x-1)/(\xi_3 - 1)} (a^m - b^m),$$
so
$$\prod_{p \text{ counted by } i} p = O\left( \prod_{m \le x^{1/2}/\log(x)} a^m\right)$$
and thus
$$\sum_{p \text{ counted by } i} \log(p) = O\left( \sum_{m \le x^{1/2}/\log(x)} m \right) = O(x/\log^2(x)).$$

Since we have a fixed number of indices $i$, this proves that the number of counted $p$ is $O(x/\log^2(x)) = o(x/\log(x))$.
\end{proof}

The last part is proving $L(x, \xi_1, \xi_2) = o(x/\log(x))$. For each $W_i$ fix some $w_i \neq -1, 0, 1$. Now for all large enough $\ell$, in particular for $\ell > \xi_1$, we have $[\mathbb{Q}(\zeta_{\ell}, w_i^{1/\ell}) : \mathbb{Q}] = \ell(\ell-1)$ by Proposition \ref{prop:kummerDisjoint}. Applying the GRH conditional version of the Chebotarev density theorem now gives

\begin{align*}
L(x, \xi_1, \xi_2) \le \sum_{\xi_1 \le \ell \le \xi_2} P(x, \ell) \\
\le \sum_{\xi_1 \le \ell \le \xi_2} \sum_{1 \le i \le t} \left( \frac{\li(x)}{[\mathbb{Q}(\zeta_{q_i(\ell)}, W_i^{1/q_i(\ell)}) : \mathbb{Q}]} + O(x^{1/2}\log(q_i(\ell)x)) \right) \\
\le \sum_{\xi_1 \le \ell \le \xi_2} \sum_{1 \le i \le t} \left(\frac{\li(x)}{\ell(\ell-1)} + O(x^{1/2}\log(\ell x))\right) \\
\le \li(x)O\left(\sum_{\xi_1 \le \ell \le \xi_2} \frac{1}{\ell^2} \right) + O(x^{1/2}\pi(\xi_2)\log(\xi_2x)) \\
= o(\li(x)) + O(x/\log(x)^2) \\
= o(x/\log(x)).
\end{align*}
This finishes the proof of $d(M) = d_{\infty}$.

\subsection{Proof of Theorem \ref{thm:densitynonzero}: $d_n \neq 0$ for all $n$ implies $d_{\infty} \neq 0$}

Let $W$ be the subgroup of $\mathbb{Q}^{\times}$ generated by all of $W_i$. By Lemma 5.6. of Lenstra \cite{Lenstra} we have $L_{\ell}$ and $L_nF$ linearly disjoint for all $\ell \nmid n$, $\ell$ large enough, as we have $L_{\ell} = \mathbb{Q}(\zeta_{\ell}, W^{1/{\ell}})$ for $\ell$ large. (Alternatively, apply Proposition \ref{prop:kummerDisjoint}.)

By this we get the product formula
$$d_{n\ell} = d_nd_{\ell}'.$$
Here $d_{\ell}'$ is the proportion of elements of $\Gal(L_{\ell}/\mathbb{Q})$ not fixing any of the subfields $\mathbb{Q}(\zeta_{q_i(\ell)}, W_i^{1/q_i(\ell)})$. Thus, for some $n$ we have
\begin{align}
\label{eq:prod}
d_{\infty} = d_n\prod_{\ell \nmid n} d_{\ell}'.
\end{align}
We are left with proving that this infinite product converges to a strictly positive value assuming no term of it is zero.

The Galois group $\Gal(L_{\ell}/\mathbb{Q})$ is, for $\ell$ large enough, of size $(\ell - 1)\ell^r$, where $r$ is the rank of $W$ (\cite{Lenstra}, Lemma 5.2, or by Proposition \ref{prop:kummerDisjoint}). The number of elements of $\Gal(L_{\ell}/\mathbb{Q})$ fixing the subfield $\mathbb{Q}(\zeta_{\ell}, W_i^{1/\ell})$ for a given $i$ is $O(\ell^{r-1})$, as this subfield has degree at least $(\ell-1)\ell$ for $\ell$ large enough. By the union bound,
$$d_{\ell}' \ge 1 - \frac{t}{\ell^2},$$
implying that the infinite series indeed is positive unless $d_{\ell}' = 0$ for some $\ell$.

We remark that in the case $t = 1$ one has
$$d_{\ell}' = 1 - \frac{1}{(\ell - 1)\ell^r}$$
for $\ell$ large enough, where $r$ is the rank of $W_1$. Therefore, for $t = 1$ the density $d_{\infty}$ is always a rational multiple of
$$\prod_{\ell} 1 - \frac{1}{(\ell - 1)\ell^r}.$$
In the general case $d_{\ell}'$ can be calculated by inclusion-exclusion in terms of the ranks of subgroups generated by a subset of $W_1, \ldots , W_t$.

\section{Proof of Theorem \ref{thm:equal}: equality of orders}

The proof of the existence of the densities in our results is postponed to Section \ref{sec:ordering}.

To demonstrate the methods we first prove the positivity part of Corollary \ref{cor:equalPositive}. Let $a_1, \ldots , a_t$ be positive rationals different from $1$. For each $i$, choose $W_i = \{a_i^n | n \in \mathbb{Z}\}$ and $k_i = 2$. Let $F = \mathbb{Q}(\sqrt{a_1}, \ldots , \sqrt{a_t})$ and $C = \{\id_F\}$. Now $M$ contains those primes $p$ for which the index of $W_i$ modulo $p$ divides $2$ for all $i$, and for which $a_i$ is a quadratic residue modulo $p$. Therefore, $p \in M$ if and only if $\ord_p(a_1) = \ldots = \ord_p(a_t) = \frac{p-1}{2}$.

The statement $d(M) \neq 0$ is, by Theorems \ref{thm:densitylimit} and \ref{thm:densitynonzero}, equivalent to the statement $d_n \neq 0$ for all $n$. That $d_n \neq 0$ for all $n$ follows by noting that the element $\sigma$ of $\Gal(FL_n/\mathbb{Q})$ which maps an element to its complex conjugate fixes the real field $F$, but does not fix any of the nonreal fields $\mathbb{Q}(\zeta_{q_i(\ell)}, W_i^{1/q_i(\ell)})$.

We then focus on the general case of Theorem \ref{thm:equal}. Let $a_1, \ldots , a_k$ be given rationals not equal to $-1, 0, 1$. Let $W_i$ be the subgroup generated by $a_i$. We will choose all $k_i$ to be equal to $3 \cdot 2^s$, where $s$ a integer chosen later. Pick $F = \mathbb{Q}(\zeta_{3 \cdot 2^s}, a_1^{1/3 \cdot 2^s}, \ldots , a_k^{1/3 \cdot 2^s})$ with $C = \{\id_F\}$.  Define $q(n) = q_1(n)$, so we have $q(2) = 2^{s+1}$, $q(3) = 9$, and $q(\ell) = \ell$ for $\ell \ge 5$. Let
$$F' = \mathbb{Q}(\zeta_{3 \cdot 2^s}, a_1^{1/2^s}, \ldots , a_k^{1/2^s}).$$

By Theorems \ref{thm:densitylimit} and \ref{thm:densitynonzero} it suffices to show $d_n \neq 0$ for all $n$.  We will first prove that there exists an element in $\Gal(F'L_2/\mathbb{Q})$ fixing $F'$ but not any of $\mathbb{Q}(\zeta_{2^{s+1}}, a_i^{1/2^{s+1}})$ under the condition of the theorem, when $s$ is chosen suitably. This is the difficult part of the proof, as the obstructions of the theorem live inside $F'L_2$. We then extend the constructed map to $FL_n$ for any $n$, this being relatively easy.

\subsection{Controlling parities of orders}

We first present a lemma.

\begin{lemma}
\label{lem:oddBasis}
Let $a_1, \ldots , a_k$ be non-zero rationals. Assume that no product of $a_i$ with (possibly negative) integer exponents equals $-1$. Then there exists a set $S \subset \{a_1, \ldots , a_k\}$ with the following properties:
\begin{itemize}
\item[(i)] The elements of $S$ are multiplicatively independent.
\item[(ii)] For all $1 \le i \le k$ there exists an odd integer $e$ such that $a_i^e$ may be expressed as a product of the elements of $S$ with (possibly negative) integer exponents.
\end{itemize}
\end{lemma}

\begin{proof}
Choose a set $S \subset \{a_1, \ldots , a_k\}$ of minimum size satisfying the second condition. (At least one such set exists, as we may pick $\{a_1, \ldots , a_k\}$.) We prove that condition (i) holds.

Assume not. Write
$$\prod_{s \in S} s^{f(s)} = 1$$
for some $f : S \to \mathbb{Z}$ which is not zero everywhere. As long as all $f(s)$ are even, take square roots. After each such operation the right hand side stays $1$, as no product of $a_i$ equals $-1$. Thus, we may assume that $t \in S$ is such that $f(t)$ is odd, and write
$$t = \prod_{t \neq s \in S} s^{-f(s)/f(t)}.$$ 
We now prove that $S' = S \setminus \{t\}$ satisfies condition (ii), which leads to a contradiction.

As $S$ satisfies condition (ii), for any $a_i$ we may write
$$a_i^e = \prod_{s \in S} s^{E(s)},$$
for some $E : S \to \mathbb{Z}$. Write this as
$$a_i^e = \prod_{t \neq s \in S} s^{-E(t)f(s)/f(t)} \prod_{t \neq s \in S} s^{E(s)}.$$
Now raise both sides to the odd power $f(t)$ to obtain a desired representation for $a_i$ in terms of $S'$.
\end{proof}

We divide into two cases according to whether or not there is some product of $a_i$ (with possibly negative exponents) equal to $-1$ or not.

\textit{Case 1. No product of the numbers $a_i$ is equal to $-1$.}

Let $\{b_1, \ldots , b_v\}$ denote a subset of the type of the proposition when applied to $a_1, \ldots , a_k$. Let $N$ be as in Proposition \ref{prop:kummerGalois} when applied to $b_1, \ldots, b_v$, and write $N = 2^t \cdot m$, where $m$ is odd. Now there exists an automorphism of
$$\mathbb{Q}(\zeta_{6N}, b_1^{1/N}, \ldots , b_v^{1/N})$$
sending $b_i^{1/N} \to b_i^{1/N}$ and $\zeta_{6N} \to -\zeta_{6N}$. By restricting this gives an automorphism $\sigma$ of
$$K = \mathbb{Q}(\zeta_{3 \cdot 2^{t+1}}, b_1^{1/2^t}, \ldots , b_v^{1/2^t})$$
which maps $b_i^{1/2^t} \to b_i^{1/2^t}$ and $\zeta_{3 \cdot 2^{t+1}} \to -\zeta_{3 \cdot 2^{t+1}}$. 

We now claim that $a_i^{1/2^t}$ belongs to $K$ for any $i$, and furthermore that $\sigma$ fixes $a_i^{1/2^t}$, meaning that we have found a map we were looking for with the choice $s = t$.

Write
$$a_i^e = \prod_{j = 1}^v b_j^{e_j}$$
with $e$ odd. Take $2^t$th roots and raise both sides to the $e^{-1} \pmod{2^t}$th power to get a rational times $a_i^{1/2^t}$ on left hand side and an element of $K$ on the right hand side, proving $a_i^{1/2^t} \in K$. By mapping both sides with $\sigma$ we get that $a^{1/2^t}$ is fixed.\\

\textit{Case 2. There is a product of the numbers $a_i$ equal to $-1$.}

We divide into two subcases depending on whether the exponents in the product equal to $-1$ have an odd or even sum. In both cases we assume that there is no product of $a_i$ with an odd number of terms equal to $1$.

\textit{Case 2.1. There is a product of an even number of $a_i$ equal to $-1$.}

Let $\{b_1, b_2, \ldots , b_v\}$ be a subset given by Lemma \ref{lem:oddBasis} when applied to the numbers $|a_i|$. Let $N$ be as in Proposition \ref{prop:kummerGalois} when applied to $b_1, \ldots , b_v$ and write $N = 2^t \cdot m$ with $m$ odd. We obtain an automorphism of
$$\mathbb{Q}(\zeta_{12N}, b_1^{1/2N}, \ldots , b_v^{1/2N})$$
fixing $\zeta_{12N}$ and mapping $b_i^{1/2N} \to -b_i^{1/2N}$. Restricting to
$$K = \mathbb{Q}(\zeta_{3 \cdot 2^{t+2}}, b_1^{1/2^{t+1}}, \ldots , b_v^{1/2^{t+1}})$$
gives a map $\sigma$ fixing $\zeta_{3 \cdot 2^{t+2}}$ and mapping $b_i^{1/2^{t+1}} \to -b_i^{1/2^{t+1}}$.

We claim that this $\sigma$ is what we want for $s = t$, assuming that there is no odd product of $a_i$ equal to $1$. This is done by proving that $a_i^{1/2^{t+1}} \to -a_i^{1/2^{t+1}}$ for all $i$.

Note first that $|a_i|^{1/2^{t+1}} \in K$ for all $i$, with the proof being the same as for the analogous claim in the previous case. As $\zeta_{2^{t+2}} \in K$, we have $a_i^{1/2^{t+1}} \in K$.

Note then that since $\zeta_{2^{t+2}}$ is fixed, we have $a_i^{1/2^{t+1}} \to -a_i^{1/2^{t+1}}$ if and only if $|a_i|^{1/2^{t+1}} \to -|a_i|^{1/2^{t+1}}$.

We now prove that $|a_i|^{1/2^{t+1}} \to -|a_i|^{1/2^{t+1}}$. By the choice of $b_i$, we may write
\begin{align}
\label{eq:equ1}
|a_i|^e = \prod_{j = 1}^v b_j^{e_j}
\end{align}
with $e$ odd. We first prove that the sum of $e_j$ is odd. Assume the contrary. We may write each term $b_j$ as $\pm a_{j'}$ for some $j'$ and choice of sign $\pm$. If the sign is minus, we may take an even number of the numbers $a_1, \ldots , a_k$ with product $-1$ in the place of the minus sign. Doing the same replacement for $|a_i|$ we obtain that there exists an odd number of $a_1, \ldots , a_k$ with product $1$, a contradiction.

Thus, the sum of $e_j$ is odd. Now, take $2^{t+1}$th roots of \eqref{eq:equ1} and map both sides by $\sigma$. The right hand side is mapped to its additive inverse, and thus $|a_i|^{e/2^{t+1}} \to -|a_i|^{e/2^{t+1}}$. Raising to the power of $e^{-1} \pmod{2^{t+1}}$ gives $|a_i|^{1/2^{t+1}} \to -|a_i|^{1/2^{t+1}}$, as desired.\\

\textit{Case 2.2. There is a product of an odd number of $a_i$ equal to $-1$.}

Let $\{b_1, b_2, \ldots , b_v\}$ be again a subset given by Lemma \ref{lem:oddBasis} for the numbers $|a_i|$. Let $x_1, x_2, \ldots , x_v \in \{-1, 1\}$ be parameters. As in the previous cases, we apply Proposition \ref{prop:kummerGalois} to $b_1, \ldots , b_v$. We get that for any choice of $x_i$ the exists an automorphism $\sigma$ of
$$K = \mathbb{Q}(\zeta_{3 \cdot 2^{t+2}}, b_1^{1/2^{t+1}}, \ldots , b_v^{1/2^{t+1}})$$
mapping $\zeta_{3 \cdot 2^{t+2}} \to -\zeta_{3 \cdot 2^{t+2}}$ and $b_i^{1/2^{t+1}} \to x_ib_i^{1/2^{t+1}}$.

Similarly to the case 2.1 we have $|a_i|^{1/2^{t+1}} \in K$ for all $i$. The difference is that this time $a_i^{1/2^{t+1}}$ is fixed if and only if $(-a_i)^{2^{t+1}} \to -(-a_i)^{1/2^{t+1}}$, as $\zeta_{2^{t+2}} \to -\zeta_{2^{t+2}}$. 

For $a_i > 0$ we may write
$$a_i^e = \prod_{j = 1}^v b_j^{e_j}$$
with $e$ odd. Similarly to before, take $2^{t+1}$th roots, raise to the $e^{-1} \pmod{2^{t+1}}$th power and map by $\sigma$. We obtain that $a_i^{1/2^{t+1}}$ is not mapped to itself if and only if a linear equation of the form
\begin{align}
\label{eq:equ3}
\sum_{j = 1}^v e_jx_j \equiv 1 \pmod{2}
\end{align}
is satisfied.

For $a_i < 0$ we proceed similarly by considering the representation of $|a_i|$ and get a linear equation of the form
\begin{align}
\label{eq:equ4}
\sum_{j = 1}^v e_jx_j \equiv 0 \pmod{2}
\end{align}
(where the $e_j$ might be different from the ones in \eqref{eq:equ3}).

By elementary linear algebra, there exists a solution unless one can take a linear combination of the equations \eqref{eq:equ4} and an odd number of the equations \eqref{eq:equ3}, resulting in $0 \equiv 1 \pmod{2}$. We prove that this is the case only if some odd number of terms $a_1, \ldots , a_k$ have product $1$. 

This situation corresponds to having
\begin{align}
\label{eq:equ2}
\prod_{a_i > 0} a_i^{f_i} \prod_{a_i < 0} |a_i|^{f_i} = t^2,
\end{align}
where $t$ is some product of the numbers $b_1, \ldots , b_v$ and
$$\sum_{a_i > 0} f_i \equiv 1 \pmod{2}.$$
Note that $t^2$ is a product of an even number of $a_i$.

Drop the absolute signs in the terms $|a_i|$ in \eqref{eq:equ2}. If the sign of the left hand side stays the same, i.e. the sum of $f_i, a_i < 0$ is even, we get that there is an odd number of $a_1, \ldots , a_k$ having product $1$, a contradiction. 

If the sign changes, i.e. the sum of $f_i, a_i < 0$ is odd, we obtain that there is an even number of $a_1, \ldots , a_k$ having product $-1$. Now we have both an odd and even number of $a_i$ having product $-1$, so we may multiply the products and obtain an odd number of $a_i$ having product $1$, a contradiction.

\subsection{Extending to $d_n \neq 0$ for all $n$}

We first extend our element in $\Gal(F'L_2/\mathbb{Q})$ to $\Gal(FL_6/\mathbb{Q})$. Note first that the degree of 
$$F'L_2 = \mathbb{Q}(\zeta_{3 \cdot 2^s}, a_1^{1/2^s}, \ldots , a_k^{1/2^s})$$
is a power of two, while the degree of
$$K = \mathbb{Q}(\zeta_9, a_1^{1/3}, \ldots , a_k^{1/3})$$
is two times a power of three. Thus, the intersection of $F'L_2$ and $K$ has degree at most two, so the intersection is $\mathbb{Q}(\zeta_3)$. 

This means that given any elements $\sigma_{F'L_2} \in \Gal(F'L_2/\mathbb{Q})$ and $\sigma_K \in \Gal(K/\mathbb{Q})$ which agree on $\mathbb{Q}(\zeta_3)$ one can find an element in the Galois group of the compositum
$$F'L_2K = \mathbb{Q}(\zeta_{9 \cdot 2^s}, a_1^{1/3 \cdot 2^s}, \ldots , a_k^{1/3 \cdot 2^s}) = FL_6$$
which restricted to $F'L_2$ gives $\sigma_{F'L_2}$ and to $K$ gives $\sigma_K$. 

We apply this with $\sigma_{F'L_2}$ being the map constructed in the previous subsection and with $\sigma_K$ mapping $\zeta_9 \to \zeta_9^4$ and fixing $a_i^{1/3}$. Such a map $\sigma_K$ exists. Indeed, by Proposition \ref{prop:kummerAbelian} the largest abelian subextension of $\mathbb{Q}(\zeta_3, a_1^{1/3}, \ldots , a_k^{1/3})$ is of the form
$$\mathbb{Q}(\zeta_3, \sqrt{b_1}, \ldots , \sqrt{b_K})$$
for some $b_i$. The degree of this extension is a power of two, while $\zeta_9$ has degree $6$ over $\mathbb{Q}$. 

The resulting combination of $\sigma_{F'L_2}$ and $\sigma_K$ fixes $F$ but does not fix any of $\mathbb{Q}(\zeta_{q(\ell)}, a_i^{1/q(\ell)})$ for $\ell \in \{2, 3\}$. Thus, we have the desired map for $FL_6$.

We finally prove that if $\ell \ge 5$, then one may extend a good element of $\Gal(FL_n/\mathbb{Q})$ to a good element of $\Gal(FL_{n\ell}/\mathbb{Q})$. The idea is the same as for $\ell = 3$: do not fix $\zeta_{q(\ell)}$. We have to check that
$$\zeta_{\ell} \not\in FL_n,$$
when $\ell \nmid n$. By Proposition \ref{prop:kummerAbelian} the largest abelian subfield of $FL_n$ is of the form
$$\mathbb{Q}(\zeta_{q(n)}, \sqrt{b_1}, \ldots , \sqrt{b_m})$$
for some $b_i \in \mathbb{Q}$.  We prove that $\zeta_{\ell}$ is not contained in any field of this type. 

We may assume that $b_i$ are distinct primes coprime with $q(n) \equiv 0 \pmod{4}$ and $b_m = \ell$. Note that
$$\mathbb{Q}(\zeta_{q(n)}, \sqrt{b_1}, \ldots , \sqrt{b_m}) \subset \mathbb{Q}(\zeta_{q(n)b_1 \cdots b_{m-1}}, \sqrt{b_m}) \subset \mathbb{Q}(\zeta_{q(n)b_1 \cdots b_{m-1}b_m}).$$
The field in the middle has degree at most $2\phi(q(n)b_1 \cdots b_{m-1})$, which is less than the degree $\phi(q(n)b_1 \cdots b_m)$ of the field on the right. Therefore  $\zeta_{\ell}$ does not belong to the field in the middle and thus not to the field on the left.

\section{Proof of Theorem \ref{thm:ordering}: Order of orders} \label{sec:ordering}

We first prove that the densities in our theorems exist. This is implied by the following two lemmas. (For Theorem \ref{thm:system} note that for $P \in \mathbb{Z}[x]$ the set of primes $p$ for which $P(x) \equiv 0 \pmod{p}$ is solvable corresponds to those $p$ having suitable Artin symbol in the splitting field of $P$.)

\begin{lemma}
\label{lem:denExist}
Assume GRH. Let $a_1, \ldots , a_t$ be rationals not equal to $-1, 0, 1$, let $k_1, \ldots , k_t$ be positive integers, let $F'$ be a finite Galois extension and let $C'$ be a conjugacy class in $\Gal(F/\mathbb{Q})$. The density of the primes $p$ satisfying $\left(\frac{F'/\mathbb{Q}}{p}\right) \in C'$ and $\ord_p(a_i) = (p-1)/k_i$  for all $i$ exists.
\end{lemma}

The lemma is proven by choosing the set $M$ in Theorem \ref{thm:densitylimit} suitably. A similar choice was made when proving the positivity part in Theorem \ref{thm:equal}, and similar choices will be made in the proofs of the other results.

\begin{proof}
Let $K$ be the least common multiple of $k_i$. Apply Theorem \ref{thm:densitylimit} with $W_i = \{a_i^n | n \in \mathbb{Z}\}$, $$F = F'(\zeta_K, a_1^{1/k_1}, \ldots , a_t^{k_t})$$
and with $C \subset \Gal(F/\mathbb{Q})$ consisting of the (possibly empty) set of elements of $\Gal(F/\mathbb{Q})$ which fix
$$\mathbb{Q}(\zeta_K, a_1^{1/k_1}, \ldots , a_t^{k_t})$$
and whose restriction to $\Gal(F'/\mathbb{Q})$ belongs to $C$.
\end{proof}

\begin{lemma}
\label{lem:bigOrders}
Assume GRH. Let $a$ be a rational with $a \neq -1, 0, 1$. For any $\epsilon > 0$ there exists $C > 0$ such that the density of the primes $p$ satisfying $\ord_p(a) \ge (p-1)/C$ is over $1 - \epsilon$.
\end{lemma}

\begin{proof}
See \cite{Wagstaff} (Section 5).
\end{proof}

We then prove Theorem \ref{thm:ordering}.

\begin{proof}
Let $h_1, \ldots , h_k$ be integers such that the density of primes $p$ with
$$\ord_p(a_i) = \frac{p-1}{h_i}$$
is positive. Such integers exist by Lemma \ref{lem:bigOrders}.

Let $\ell_1 = 1$. For each $i = 2, 3, \ldots , k$ we do the following. Let $\ell_i$ be a large prime, large enough such that
$$\ell_ih_i > \ell_{i-1}h_{i-1}$$
and that the extension 
$$K_i = \mathbb{Q}(\zeta_{\ell_i^2}, a_1^{1/\ell_i}, \ldots , a_k^{1/\ell_i})$$
is independent of ``everything else'', i.e. it is linearly disjoint with fields of the form
$$\mathbb{Q}(\zeta_n, a_1^{1/n}, \ldots , a_k^{1/n})$$
with $\ell_i \nmid n$. The existence of such $\ell_i$ is guaranteed by Proposition \ref{prop:kummerDisjoint}. In fact, any large enough $\ell_i$ suits. 

We first prove that for $\ell_i$ large enough, there exists an automorphism of $K_i$ fixing $\zeta_{\ell_i}$ and $a_i^{1/\ell_i}$ which does not fix $\zeta_{\ell_i}^2$ nor any of $a_j^{1/\ell_i}, j \neq i$. By Proposition \ref{prop:kummerDisjoint}, we know that the field $\mathbb{Q}(\zeta_{\ell_i}, a_i^{1/\ell_i}, a_j^{1/\ell_i})$ has degree $(\ell_i - 1)\ell_i^2$ for $\ell_i$ large enough by the multiplicative independence of $a_i$ and $a_j$. Thus, exactly one of the $\ell_i$ maps fixing $\zeta_{\ell_i}$ and $a_i^{1/\ell_i}$ also fixes $a_j^{1/\ell_i}$.

Therefore, a (uniformly) random isomorphism of $K_i$ fixing $\zeta_{\ell_i}$ and $a_i^{1/\ell_i}$ fixes at least one of $a_j^{1/\ell_i}, j \neq i$ with probability approaching zero as $\ell_i \to \infty$. One similarly sees $\zeta_{\ell_i^2}$ is mapped to itself with probability approaching zero. This proves the existence of the desired automorphism of $K_i$ for $\ell_i$ largen.

We then note that as the local densities corresponding to the primes $p$ satisfying
$$\ord_p(a_j) = \frac{p - 1}{h_jq_j}, j < i, \ord_p(a_j) = \frac{p-1}{h_j}, j \ge i$$
is by assumption positive, then the local densities corresponding to
$$\ord_p(a_j) = \frac{p-1}{h_jq_j}, j \le i, \ord_p(a_j) = \frac{p-1}{h_j}, j > i$$
are positive, too. This follows from the fact that $K_i$ is linearly disjoint with the other Kummer-type fields arising in the local densities. By Theorems \ref{thm:densitylimit} and \ref{thm:densitynonzero} this implies the result.
\end{proof}

\section{Proof of Theorem \ref{thm:system}: Solvability of system of equations}

We have already in Section \ref{sec:ordering} proven that the density exists, so we focus on positivity.

Let $K$ be the compositum of splitting fields of $P_i$ and let $C$ consist of the identity of $\Gal(K/\mathbb{Q})$. Let $p_1, \ldots , p_t$ be the primes dividing some numerator or denominator of $a_i, b_i, i = 1, \ldots , k$. Let $N$ be the integer obtained by applying Proposition \ref{prop:kummerGalois} to the numbers $p_i$ with the field $K$.

We prove that the density of the primes $p$ with $\ord_p(a_1) = \ord_p(b_1) = \ldots = \ord_p(b_k) = (p-1)/N$ and $\left(\frac{K/\mathbb{Q}}{p}\right) \in C$ is positive. This is clearly sufficient. We apply Theorems \ref{thm:densitylimit} and \ref{thm:densitynonzero} in the obvious way, requiring that $p$ splits in
$$F = K(\zeta_N, a_1^{1/N}, \ldots , b_k^{1/N}).$$
We now only have to check that the local densities are nonzero.

Let $n$ be a positive squarefree integer. By the choice of $N$, there exists an isomorphism $\sigma$ of
$$K_n := K(\zeta_{q(n)}, p_1^{1/q(n)}, \ldots , p_t^{1/q(n)})$$
which fixes
$$K(\zeta_{N}, p_1^{1/N}, \ldots, \ldots , p_k^{1/N})$$
but does not fix $\zeta_{q(\ell)}$ for any $\ell$. Since $a_i, b_i > 0$, the field $K_n$ contains
$$FL_n = \mathbb{Q}(\zeta_{q(n)}, a_1^{1/q(n)}, b_1^{1/q(n)}, \ldots, b_k^{1/q(n)}),$$
and the positivity of the local density $d_n$ follows by restricting $\sigma$ to $FL_n$.

\section{Discussion} \label{sec:discussion}

The developed methods reduce a large class of problems concerning the reductions of multiplicative subgroups of $\mathbb{Q}^{\times}$ modulo primes to analyzing local obstructions in Kummer-type extensions. In our applications the obstructions were rather straightforward, the most difficult ones faced being those in Theorem \ref{thm:equal}. However, there are problems with much more complex local obstructions. We mention a couple of such problems.

Our first example is a natural generalization of Theorem \ref{thm:ordering}: determine all rationals $a_i$ such that we have $\ord_p(a_1) > \ldots > \ord_p(a_k)$ for infinitely many $p$.

For the sake of discussion, assume all $a_i$ are positive. The chain of inequalities $\ord_p(a_1) > \ldots > \ord_p(a_k)$ can be seen as a collection of chains of the form $\ord_p(b_i^{e_{i, 1}}) > \ord_p(b_i^{e_{i, 2}}) > \ldots > \ord_p(b_i^{e_{i, l_i}})$ interlacing with each other, where $b_i$ are pairwise multiplicatively independent and none of $b_i$ is a perfect power. For each $i$ we have to decide on the divisibility of the order of $b_i$ by various prime powers to satisfy the inequalities $\ord_p(b_i^{e_{i, j}}) > \ord_p(b_i^{e_{i, j+1}})$. In addition to this, in all cases it is not possible to guarantee that some $\ord_p(b_i)$ are divisible by a given integer and some are not, and we also have to make sure that we may combine the chains $\ord_p(b_i^{e_{i, 1}}) > \ldots > \ord_p(b_i^{e_{i, l_i}})$. 

Another related problem is considering the solvability of a system of equations of the form $a_i^{x_i} \equiv b_i \pmod{p}$, $1 \le i \le k$. The solvability is equivalent to the conditions $\ord_p(a_i) | \ord_p(b_i)$. Such equations are the topic of the two-variable Artin conjecture (see \cite{Moree-Steven}). Theorem \ref{thm:system} already proves that such a system is solvable for $a_i, b_i > 1$ infinitely often, but this is in general not the case. Note that by choosing the cyclic system $b_i = a_{i+1}$ we arrive at Theorem \ref{thm:equal}.

By similar work as in the proof of Theorem \ref{thm:equal}, the solvability of a system of exponential equations can be reduced to considering the $2$-adic valuations of orders. Again, determining the local obstructions is rather unpleasant, and the necessary and sufficient condition is not very enlightening.

We conclude by mentioning a result from \cite{Moree-Steven}: for multiplicatively independent $a$ and $b$, the density of primes $p$ such that the equation $a^x \equiv b \pmod{p}$ is solvable is a positive rational number times a universal constant. The product formula \ref{eq:prod} is not enough to give product formulas for densities of primes defined by conditions such as the solvability of $a^x \equiv b \pmod{p}$, $\ord_p(2) = \ord_p(3)$ or $\ord_p(2) > \ord_p(3)$. It would be of interest to obtain results similar to those in \cite{Moree-Steven} for a wider class of problems.

\end{document}